\documentclass[12pt]{amsart}
\usepackage[english]{babel}
\usepackage[T1]{fontenc}
\usepackage{geometry}
\usepackage[dvips]{graphicx}
\usepackage{fullpage}
\usepackage{amsmath,amscd,amsthm,amsfonts,latexsym,epsfig}
\usepackage{color}

\theoremstyle{plain}
\newtheorem{theorem}                 {Theorem}      [section]
\newtheorem{proposition}  [theorem]  {Proposition}

\newtheorem*{theorem*}{Theorem}
\theoremstyle{definition}
\newtheorem{example}      [theorem]  {Example}
\newtheorem{remark}       [theorem]  {Remark}

\def \imto{\hookrightarrow}

\newcommand{\p}{\partial}

\def \r{\mbox{${\mathbb R}$}}

\def \h{\mbox{${\mathbb H}$}}
\newcommand{\df}{\,\mathrm{d}}
\def\E{\mathbb{E}}
\def \s{\mbox{${\mathbb S}$}}

\def \ri{\mbox{$\mathrm{Ricc}$}}

\def \ln{\mbox{${\overline{{\nabla}}}$}}
\DeclareMathOperator{\grad}{grad}

\DeclareMathOperator{\trace}{trace}

\begin{document}

\title{Biharmonic constant mean curvature surfaces in Killing submersions}

\author{Stefano Montaldo}
 
 \address{Universit\`a degli Studi di Cagliari\\
 Dipartimento di Matematica e Informatica\\
 Via Ospedale 72\\
 09124 Cagliari}
 \email{montaldo@unica.it}
 
\author{Irene I. Onnis}
\address{Departamento de Matem\'{a}tica, C.P. 668\\ ICMC,
USP, 13560-970, S\~{a}o Carlos, SP\\ Brasil}
\email{onnis@icmc.usp.br}

\author{Apoena Passos Passamani}
\address{Departamento de Matem\'{a}tica, UFES, 29075-910, Vit\'{o}ria, ES, Brasil}
	\email{apoena.passamani@ufes.br}

 \thanks{
 The first author was supported by P.R.I.N. $2015$ -- Real and Complex Manifolds: Geometry, Topology and Harmonic Analysis -- Italy. The second author was supported by grant 2016/24707-4, 
 S\~ao Paulo Research Foundation (Fapesp).}


\begin{abstract}
A $3$-dimensional Riemannian manifold is called Killing submersion if it admits a Riemannian submersion over a surface such that its fibers are the trajectories of a complete unit Killing vector field. In this paper, we give a characterization of proper biharmonic CMC surfaces in a Killing submersion. In the last part, we also classify the proper biharmonic Hopf cylinders in a Killing submersion. 
\end{abstract}

\maketitle

\section{Introduction}
{\it Harmonic maps}  are critical points of the {\em energy} functional
\begin{equation}\label{energia1}
E(\varphi)=\frac{1}{2}\int_{N}\,|d\varphi|^2\,dv_g \,\, ,
\end{equation}
where $\varphi:(N,g)\to(\overline{N},h)$ is a smooth map between two Riemannian
manifolds $N$ and $\overline{N}$. In analytical terms, the condition of harmonicity is equivalent to the fact that the map $\varphi$ is a solution of the Euler-Lagrange equation associated to the energy functional \eqref{energia1}, i.e.
\begin{equation}\label{harmonicityequation}
    {\trace} \, \nabla d \varphi =0 \,.
\end{equation}
The left member of \eqref{harmonicityequation} is a vector field along the map $\varphi$ or, equivalently, a section of the pull-back bundle $\varphi^{-1} \, (T\overline{N})$: it is called {\em tension field} and denoted by $\tau (\varphi)$. 
We refer to \cite{PBJCW,EL,Xin} for notations and
background on harmonic maps.

A related topic of growing interest deals with the study of the so called {\it biharmonic maps}: these maps, which provide a natural generalization of harmonic maps, are the critical points of the bienergy functional (as suggested by Eells--Lemaire in \cite{EL}):
\begin{equation}\label{bienergia}
    E_2(\varphi)=\frac{1}{2}\int_{N}\,|\tau (\varphi)|^2\,dv_g \,\, .
\end{equation}
In \cite{J1}, Jiang derived the first and the second variation formulas for the bienergy. In particular, he showed that the Euler-Lagrange equation associated to $E_2(\varphi)$ is given by
\begin{equation}\label{bitensionfield1}
    \tau_2(\varphi) = - \triangle \tau(\varphi)- \trace \overline{R}(d \varphi, \tau(\varphi)) d \varphi = 0 \,\, ,
    \end{equation}
where  $\triangle$ is the rough Laplacian defined on sections of $\varphi^{-1} \, (T\overline{N})$ and
\begin{equation}\label{curvatura}
    \overline{R} (X,Y)= \ln_X \ln_Y - \ln_Y \ln_X -\ln_{[X,Y]}
\end{equation}
is the curvature operator on $(\overline{N},h)$.
Equation~\eqref{bitensionfield1} is a {\it fourth order}  semi-linear elliptic system of differential equations. We also note that any harmonic map is an absolute minimum of the bienergy and, so, it is trivially biharmonic. Therefore, a general working plan is to study the existence of {\it proper} biharmonic maps, i.e. biharmonic maps which are not harmonic. We refer to \cite{C1,Montaldo} for existence results and general properties of biharmonic maps.
\\


In this paper, we study biharmonic surfaces in the total space of a \textit{Killing submersion}. A Riemannian submersion $\pi: \E\rightarrow M$ of a $3$-dimensional Riemannian manifold $\E$  over a surface $M$ will be called a  Killing submersion if its fibers are the trajectories of a complete unit Killing vector field~$\xi$. Killing submersions are determined by two functions on $M$, its \textit{Gaussian curvature} $G$ and the so called \textit{bundle curvature} $r$ (for more details, see \cite{Espinar,Manzano} and Section~\ref{killing}) and, for this reason, we denote them by $\E(G,r)$. A remarkable family of Killing submersions are the homogeneous  $3$-spaces. 
In fact, with the exception of the hyperbolic $3$-space $\h^3(-1)$, simply-connected homogeneous Riemannian $3$-manifolds with isometry group of dimension $4$ or $6$ can be represented by a $2$-parameter family $\E(c,b)$, where $c,b\in\r$. These $\E(c,b)$-spaces are $3$-manifolds admitting a global unit Killing vector field whose integral curves are the fibers of a certain Riemannian submersion over the simply-connected constant Gaussian curvature surface $M(c)$ and, therefore, they determine Killing submersions $\pi:\E(c, b)\to M(c)$ (for more details, see~\cite{Da}). A local description of these examples can be given by using the so called \textit{Bianchi-Cartan-Vranceanu} spaces (see Section~\ref{killing}). As it turns out, the $\E(c,b)$-spaces are the only simply-connected homogeneous $3$-manifolds admitting the structure of a Killing submersion (see \cite{Manzano}).\\

Our work find inspiration from the paper \cite{OW}, where Ou and Wang gave a complete classification of constant mean curvature proper biharmonic surfaces in $3$-dimensional  Bianchi-Cartan-Vranceanu spaces. Following their ideas we give the following characterisation of proper biharmonic CMC surfaces in a Killing submersion which represents the main result of the paper.

\begin{theorem}\label{main-theorem}
Let $S$ be a proper biharmonic CMC surface in a Killing submersion $\E(G,r)$ and let $\phi$ be the angle between the unit normal vector field $\eta$ of $S$ and the Killing vector field $\xi$. 
\begin{itemize}
\item[a)] Assume that $\phi\equiv \pi/2$ on $S$. Then, $G$ and $r$ are constant along $S$, with $G\neq 4r^2$,  and the surface is locally a Hopf tube over a curve with constant geodesic curvature $\kappa^2_g=G-4r^2$.
\item[b)] Assume that $\phi\neq\pi/2$ at every point of $S$, then:
\begin{itemize}
\item[b1)] either $\grad(r)\equiv 0$ on $S$ and $S$ is an open set of $\s^2(1/\sqrt{2r^2})$ in $\s^3(1/|r|)$;
\item[b2)] or $\grad(r)\neq 0$ at every point of $S$ and, in this case, $4r^2-G\neq 0$ on $S$, the angle $\phi$ is determined by $\tan(2\phi)=2|\grad(r)|/(4r^2-G)$ and the shape operator $A$ of the surface satisfies
\begin{equation}\label{A-phi}
2|A|^2=\tan \phi\, \Delta \phi + |\grad(\phi)|^2.
\end{equation}  
\end{itemize}
\end{itemize}
\end{theorem}

\begin{remark}
From Theorem~\ref{main-theorem}, by a case by case inspection, we easily deduce that a CMC proper biharmonic surface $S$ in a Killing submersion $\E(G,r)$ with constant angle $\phi$ must be an Hopf tube, that is $\phi=\pi/2$.
\end{remark}

\section{Killing submersions}\label{killing}

A Riemannian submersion $\pi: \E\rightarrow M$ of a $3$-dimensional Riemannian manifold $\E$  over a surface $M$ is called a {\it Killing submersion} if its fibers are the trajectories of a complete unit Killing vector field $\xi$ (for more details, see \cite{Espinar, Manzano}). Since $\xi$ has constant norm, the fibers are geodesics in $\E$ and they form a foliation called the {\em vertical foliation}, denoted by $\mathcal{F}$. Most of the geometry of a Killing submersion is encoded in a pair of functions: ${ G},\, { r}$. While the first one,  ${G}$, represents the Gaussian curvature function of the base surface $M$, the other one, $r$, denotes the so called \emph{bundle curvature} which is defined as follows. Since $\xi$ is a (vertical) unit Killing vector field, then it is clear that for any vector field, $Z$, on $\E$, there exists a function ${r}_Z$ (which, a priori, depends on the chosen vector field) such that $\overline{\nabla}_Z \xi={ r}_Z\, Z\wedge\xi$, where $\overline{\nabla}$ denotes the Levi-Civita connection in $\E$. Actually, it is not difficult to see that ${ r}_Z$  does not depend on the vector field $Z$ (see \cite{Espinar}),
 so we get a function $r\in C^{\infty}(\E)$, the bundle curvature, satisfying
\begin{equation} \overline{\nabla}_Z \xi={ r}\, Z\wedge\xi\,. \label{bundle-curvature}
\end{equation}
The bundle curvature is obviously constant along the fibers and, consequently, it can be seen as a function on the base, ${r}\in C^{\infty}(M)$. In the product spaces $M\times\r$ the projection over the first factor is a Killing submersion, so its bundle curvature is ${r}\equiv 0$. More generally, using \eqref{bundle-curvature},  it is easy to deduce that $r\equiv 0$ in a Killing submersion if and only if the horizontal distribution in the total space is integrable. From now on, a Killing submersion will be denoted by $\E(G,r)$.

The existence of a Killing submersion over a simply connected surface $M$, with a prescribed \textit{bundle curvature}, $r\in
C^{\infty}(M)$, has been proved in \cite{Manzano}. Uniqueness, up to isomorphisms, is guaranteed under the assumption of simply connectedness for the total space also.

Moreover, the following result provides the existence of Killing submersions with prescribed bundle curvature over arbitrary Riemannian surfaces.
\begin{proposition}\cite{BarGar-preparation}\label{pro:existence-killing-submersion}
 Let $M$ be a Riemannian surface and choose any ${ r}\in C^{\infty}(M)$. Then, there exists a Killing submersion over $M$ with bundle curvature ${ r}$. In particular, it can be chosen to have compact fibers.
 \end{proposition}

Any Killing submersion is locally isometric to one of the following canonical examples (see \cite{Manzano}) which include, as we will show later, the so called {\em Bianchi-Cartan-Vranceanu spaces} for suitable choices of the functions $\lambda,a,b$.

\begin{example}[Canonical examples]\label{defi:canonical}
Given an open set $\Omega\subset\r^2$ and $\lambda,a,b\in\mathcal{C}^\infty(\Omega)$ with $\lambda>0$, the Killing submersion
\[\pi:(\Omega\times\r,\df s^2_{\lambda,a,b})\to (\Omega,\df s_\lambda^2),\quad\pi(x,y,z)=(x,y),\]
where
\begin{equation}
\df s^2_{\lambda,a,b}=\lambda^2(\df x^2+\df y^2)+(\df z-\lambda(a\df x+b\df y))^2\label{eqn:metrica-general-killing}
\end{equation}
and
$$
\df s^2_{\lambda}=\lambda^2(\df x^2+\df y^2)\,,
$$
will be called the canonical example associated to $(\lambda,a,b)$. The bundle curvature and the Gaussian curvature are given by:
\begin{equation}\label{eq:r-in-canonical-examples}
2 { r}= \frac{1}{\lambda^2}\left( (\lambda b)_x-(\lambda a)_y \right)\,,\quad { G}=- \frac{1}{\lambda^2} \Delta_o (\log \lambda)\,,
\end{equation}
where $\Delta_o$ represents the Laplacian with respect to the standard metric in the plane.\\

Let $\{e_1,e_2\}$ be the orthonormal frame in $(\Omega,\df s_\lambda^2)$, where $e_1=\frac{1}{\lambda}\p_x$ and $e_2=\frac{1}{\lambda}\p_y$, and let $\{E_1,E_2\}$ be the horizontal lift of $\{e_1,e_2\}$ with respect to $\pi$ and $E_3=\partial_z$. Since $\pi$ is the projection over the first two variables, there exist $a,b\in\mathcal{C}^\infty(\Omega)$ such that
\begin{equation}\label{eqn:base-universal-killing}
\left\{\begin{array}{l}
(E_1)_{(x,y,z)}=\tfrac{1}{\lambda(x,y)}\partial_x+a(x,y)\partial_z,\\
(E_2)_{(x,y,z)}=\tfrac{1}{\lambda(x,y)}\partial_y+b(x,y)\partial_z,\\
(E_3)_{(x,y,z)}=\partial_z.
\end{array}\right.
\end{equation}
Note that $\{E_1,E_2,E_3\}$ is an orthonormal frame in $(\Omega\times\r,\df s^2)$ which can be supposed positively oriented after possibly swapping $e_1$ and $e_2$. Now it is clear that the global frame~\eqref{eqn:base-universal-killing} is orthonormal for $\df s^2$ if and only if $\df s^2$ is the metric given by~\eqref{eqn:metrica-general-killing}. Regardless of the values of the functions $a,b\in\mathcal{C}^\infty(\Omega)$, the Riemannian metric given by equation~\eqref{eqn:metrica-general-killing} satisfies that $\pi$ is a Killing submersion over $(\Omega,\df s_\lambda^2)$.

Equation~\eqref{eqn:base-universal-killing} defines a global orthonormal frame $\{E_1,E_2,E_3\}$ for $\df s^2_{\lambda,a,b}$, where $E_1$ and $E_2$ are horizontal, and $E_3$ is a unit vertical Killing field. It is easy to check that $[E_1,E_3]=[E_2,E_3]=0$ and
$$[E_1,E_2]=\frac{\lambda_y}{\lambda^2}E_1-\frac{\lambda_x}{\lambda^2}E_2+\left(\frac{1}{\lambda^2}(b\lambda_x-a\lambda_y)+\frac{1}{\lambda}(b_x-a_y)\right)E_3.$$
\end{example}

\begin{example}[Bianchi-Cartan-Vranceanu spaces]
Particularizing the above construction, one can get models for all Killing submersions over $\r^2$, $\h^2(c)$ and the punctured sphere $\s_*^2(c)$. Given $c\in\r$, we define $\lambda_c\in\mathcal{C}^\infty(\Omega_c)$ as \[\lambda_c(x,y)=\left(1+\tfrac{c}{4}(x^2+y^2)\right)^{-1},\]
where
\[\Omega_c=\begin{cases}\{(x,y)\in\r^2:x^2+y^2<\frac{-4}{c}\}, &\text{if }c<0,\\\r^2\, ,&\text{if }c\geq 0.\end{cases}\]
Then, the metric $\lambda_c^2(\df x^2+\df y^2)$ in $\Omega_c$ has constant Gaussian curvature ${ G}=c$. If, in addition, we choose $a=-\mu\, y$ and $b=\mu\, x$  for some real constant $\mu$, then one  obtains the metrics of the Bianchi-Cartan-Vranceanu spaces $\E(c,\mu)\equiv\Omega_c\times\r$ { (}\cite[Section~2.3]{Da}{)}:
\[\lambda_c^2(\df x^2+\df y^2)+(\df z+\mu\,\lambda_c\,(y\df x-x\df y))^2.\]
A simple computation, using \eqref{eq:r-in-canonical-examples}, gives $\mu={ r}$. Thus, the Bianchi-Cartan-Vranceanu spaces can be seen as the canonical models of Killing submersions with constant bundle curvature and constant Gaussian curvature.
\end{example}

Now, we want to compute the Riemannian curvature $\overline{R}$ of the total space $\E({ G},{ r})$ of a Killing submersion $\pi:\E({ G},{ r})\to M$ in terms of ${ G}$ and the bundle curvature ${ r}$. Since the computation is purely local, we will work in a canonical example (see Example~\ref{defi:canonical}) associated to some functions $\lambda,a,b\in \mathcal{C}^\infty(\Omega)$, with $\lambda>0$ and $\Omega\subset\r^2$. Koszul formula yields the Levi-Civita connection in the canonical orthonormal frame $\{E_1,E_2,E_3\}$ given by~\eqref{eqn:base-universal-killing}:
\begin{equation}\label{eqn:killing-conexion}
\begin{array}{lll}
\overline\nabla_{E_1}E_1=-\frac{\lambda_y}{\lambda^2}E_2,&\overline\nabla_{E_1}E_2=\frac{\lambda_y}{\lambda^2}E_1+{ r} E_3,&\overline\nabla_{E_1}E_3=-{ r} E_2,\\
\overline\nabla_{E_2}E_1=\frac{\lambda_x}{\lambda^2}E_2-{ r} E_3,&\overline\nabla_{E_2}E_2=-\frac{\lambda_x}{\lambda^2}E_1,&\overline\nabla_{E_2}E_3={ r} E_1,\\
\overline\nabla_{E_3}E_1=-{ r} E_2,&\overline\nabla_{E_3}E_2={ r} E_1,&\overline\nabla_{E_3}E_3=0.
\end{array}
\end{equation}

Moreover, the frame $\{ E_1, E_2\}$ is a basis for the smooth \textit{horizontal distribution} $\mathcal{H}$ of $\E({ G}, { r})$. The horizontal projection of a vector field $U$ onto $\mathcal{H}$ will be denoted by $U^h$. A vector $U$ is called \textit{horizontal} if $U = U^h$. A \textit{horizontal curve} is a $C^{\infty }$ curve whose tangent vector lies in the horizontal distribution. Recall that  $\mathcal{H}$ is integrable if and only if ${ r}=0$. Denote by
\begin{eqnarray}
J:\mathcal{H}&\rightarrow& \mathcal{H}\, ,
\end{eqnarray}
the anticlockwise ${\pi}/{2}$ rotation. Then, from \eqref{eqn:killing-conexion}, we have
\begin{equation}
\overline{\nabla}_Z \xi={ r} (Z\wedge\xi)= -{ r} J(Z^h)\, , \quad \langle X, J(Y^h)\rangle=-\langle Y, J(X^h)\rangle\, .  \label{j}
\end{equation}
Another  direct computation from \eqref{eq:r-in-canonical-examples} and \eqref{eqn:killing-conexion} gives
\begin{eqnarray}
&\langle\overline{R}(E_j,E_3)E_j,E_3\rangle=-{ r^2}\, , \quad &\langle\overline{R}(E_1,E_2)E_j,E_3\rangle=-E_j({ r})\, , \quad j=1,2\, ,\label{R1}\\
&\langle\overline{R}(E_1,E_3)E_2,E_3\rangle= 0\, , \quad & \langle\overline{R}(E_1,E_2)E_1,E_2\rangle= 3 { r^2}- { G}\, , \label{R2}
\end{eqnarray}
where ${ G}$ is the Gaussian curvature of $M$ and, in computing the last formula, we have used that $\langle [E_1,E_2],E_3\rangle=2{ r}$. Finally, putting $X={X}^h+\langle X, E_3\rangle E_3$, $Y={Y}^h+\langle Y, E_3\rangle E_3$,  $Z={Z}^h+\langle Z, E_3\rangle E_3$,  $W={W}^h+\langle W, E_3\rangle E_3$, using \eqref{R1}, \eqref{R2} and after a long computation, one has
\begin{eqnarray}\label{Curv}
\langle\overline{R}(X,Y)Z,W\rangle &=& ({ G}-3{ r}^2)\left\{\langle Y,Z\rangle \langle X, W\rangle-\langle X,Z\rangle \langle Y,W\rangle\right\}\nonumber \\
& - &({ G}-4 { r}^2)\left\{ \langle Y,E_3 \rangle\langle Z,E_3 \rangle \langle X, W\rangle  - \langle X,E_3 \rangle\langle Z,E_3 \rangle \langle Y,W\rangle\right. \nonumber \\
   & + & \left. \langle X,E_3 \rangle\langle Y,Z \rangle\langle E_3, W\rangle - \langle Y,E_3 \rangle\langle X,Z \rangle \langle E_3, W\rangle \right\} \label{curvature}\\
  & + & \langle Z ,J(W^h) \rangle J((X\wedge Y)^h)({ r}) + \langle X ,J(Y^h) \rangle J((Z\wedge W)^h)({ r}) \, . \nonumber
\end{eqnarray}
Finally, from \eqref{R1} and \eqref{R2} we obtain the components of the Ricci curvature tensor: 
\begin{equation}\label{ricci1}
\begin{array}{lll}
\ri(E_1,E_1)=G-2r^2, & \ri(E_2,E_2)=G-2r^2, & \ri(E_3,E_3)=2r^2,\\
\ri(E_1,E_2)=0, & \ri(E_1,E_3)=-\dfrac{r_y}{\lambda}, & \ri(E_2,E_3)=\dfrac{r_x}{\lambda}.
\end{array}
\end{equation}

\section{Fundamental equations for immersed surfaces into $\E(G,r)$}\label{sec-fund}

In \cite{Da}, Daniel studies surfaces of BCV-spaces using the Killing submersion structure of these spaces and emphasizing the importance of the angle $\phi$ between the normal vector field $\eta$ of the surface and the vertical Killing vector field  $\xi=\frac{\partial}{\partial z}$. In particular, he obtained the Gauss and Codazzi equations in terms of the function $\nu=\cos \phi$. In this section, we follow these ideas to obtain similar expressions of the Gauss and Codazzi equations for a surface immersed into a Killing submersion.\\

Consider a simple connected and oriented surface $S^2$ in  $\E(G,r)$. Let $\nabla$ denote the Levi-Civita connection of $S$, $\eta$ its unit normal vector field, $A$ the shape operator associated to $\eta$ and $K$ the Gaussian curvature of $S$. Let $\phi$ be the angle function between $\eta$ and $\xi$ defined by 
$$ \langle\xi,\eta\rangle=\cos\phi.$$

The projection of $\xi$ over the tangent plane of  $S$ yields 
\begin{equation}\label{decompositionXi}
\xi=T+\cos \phi\, \, \eta,
\end{equation}
where $T$ is the tangential component of $\xi$ that satisfies $\langle T,T \rangle=\sin^2 \phi$. Then $\{T, \eta\wedge T\}$ is a local orthogonal tangent frame on $S$ and, for convenience, in the sequel we shall use its orthonormalization (under the assumption that $\phi\neq 0$):
\begin{equation}\label{basee_i}
e_1=\frac{T}{\sin\phi}, \qquad e_2=\frac{\eta\wedge T}{\sin\phi}.
\end{equation}
Moreover, with respect to the adapted frame $\{e_1,e_2,\eta\}$, the Killing vector field $\xi$ becomes
	\begin{equation}\label{xi}
     \xi=\sin \phi\, e_1+ \cos \phi \, \eta.
	\end{equation}	

Now, we can enunciate the following proposition.
\begin{proposition}\label{GaussCodazzi}
	The Gauss and Codazzi equations for an immersed surface into $\E(G,r)$ with respect to the basis \eqref{basee_i} are given, respectively, by:
	\begin{equation}\label{gauss1}
	K= \det A+ r^2+(G-4 r^2)\,\cos^2 \phi-\sin(2 \phi) \,e_2(r),
	\end{equation}
	\begin{equation}\label{codazzi1}
	\begin{aligned}
	\nabla_{e_1} A (e_2) -\nabla_{e_2} A(e_1)-A[e_1,e_2] =& \big[( 4 r^2-G)\,\cos\phi\,\sin\phi -\cos(2 \phi)\,e_2(r)   \big]\,e_2\\
	&-e_1(r)\, e_1.
	\end{aligned}
	\end{equation}		
\end{proposition}
\begin{proof}
	For the Gauss equation, we recall that
	$$K=\det A+\langle \overline{R}(e_1,e_2)\,e_2,e_1\rangle.$$
	The term $\langle \overline{R}(e_1,e_2)\,e_2,e_1\rangle$ can be compute by using  \eqref{Curv}, which confers 
	$$	K= \det A+ r^2+(G-4 r^2)\cos^2 \phi+2\cos\phi \,J (\eta^h)(r), $$
	where $J (\eta^h)=-\sin\phi\,e_2$.
	So, we obtain \eqref{gauss1}.	
	In what concerns the Codazzi equation
	$$\big\langle \nabla_{e_1} A (e_2) -\nabla_{e_2} A (e_1)-A[e_1,e_2],Z\big \rangle= \langle \overline{R}(e_1,e_2)Z,\eta \rangle,$$
	from \eqref{Curv} it results that
	\begin{equation}\label{codazzi-components}
	\begin{cases}
	\big\langle \nabla_{e_1} A(e_2) -\nabla_{e_2} A(e_1)-A[e_1,e_2],e_1\big \rangle=&\cos \phi \, J(e_2^h)(r),\\
	\big\langle\nabla_{e_1} A(e_2) \nabla_{e_2} A(e_1)-A[e_1,e_2],e_2\big \rangle=&-(G-4r^2)\cos\phi \sin\phi \\&-J(\eta^h)(r)\sin\phi-J(e_1^h)(r)\cos\phi.
	\end{cases}
	\end{equation}
  Finally, using \eqref{j} it follows that $J(e_1^h)=\cos\phi\,e_2$ and $\cos\phi\,J(e_2^h)(r)=-e_1(r)$ that, substituted in \eqref{codazzi-components}, gives immediately \eqref{codazzi1}.
\end{proof}

To determine the shape operator and the Levi-Civita connection of a surface immersed into a Killing submersion $E(G,r)$, we need the following fact.
\begin{proposition}\label{compatibilidade}
	For any vector field $X$ tangent to a surface $S$ of a Killing submersion $\E(G,r)$, it holds that:
	\begin{equation}\label{prima}
	\nabla_X T=\cos \phi\,\big(A(X)-r\, \eta\wedge X\big),
	\end{equation}
	and
	\begin{equation}\label{seconda}
	\langle A(X)-r \, \eta\wedge X,T\rangle=-X(\cos\phi).
	\end{equation}
	
\end{proposition}
\begin{proof}
	On one hand, we have
	\begin{equation}\label{nabla-x-xi1}
	\begin{aligned}
	\ln_X \xi=&\ln_X(T+\cos \phi \,\, \eta)\\
	=&\nabla_X T+\, \langle A(X),T\rangle\, \eta+X(\cos\phi)\,\eta-\cos\phi A(X).
	\end{aligned}
	\end{equation}
	On the other hand,  from \eqref{bundle-curvature}
	\begin{equation}\label{nabla-x-xi2}
	\begin{aligned}
	\ln_X \xi
	=& r\, X\wedge \xi\\
	=& r\,(\langle \eta\wedge X,T\rangle\, \eta-\cos\phi \,\eta\wedge X).
	\end{aligned}
	\end{equation}
	The result follows comparing the tangent and the normal parts of \eqref{nabla-x-xi1} and \eqref{nabla-x-xi2}.
\end{proof}
Using Proposition~\ref{compatibilidade}, the matrix associated to the shape operator $A$, with respect to the basis \eqref{basee_i}, takes the following form:
\begin{equation}\label{A1}
A=\begin{pmatrix}
e_1(\phi) & e_2(\phi)- r \\
\ \ & \ \  \\
e_2(\phi)- r & \mu \\
\end{pmatrix},
\end{equation}
where, denoting by $H=\trace A$ the mean curvature function of $S$, we deduce that
$$\mu:=\langle A(e_2), e_2\rangle=  H-e_1(\phi).$$
Moreover, using again Proposition~\ref{compatibilidade}, the Levi-Civita connection $\nabla$ of $S$ is given by:
\begin{equation}\label{nabla-ei}
\begin{aligned}
&\nabla_{e_1} e_1=\cot\phi \,(e_2(\phi)-2 r)\, e_2 ,\qquad
\nabla_{e_2} e_1=\mu \cot\phi \, e_2,\\
&\nabla_{e_1} e_2=-\cot\phi\, (e_2(\phi)-2 r)\, e_1,
\qquad \nabla_{e_2} e_2=-\mu\, \cot\phi \, e_1.
\end{aligned}
\end{equation}

\section{Biharmonic constant mean curvature surfaces}
In this section, we investigate  proper biharmonic constant mean curvature (CMC) surfaces in the total space of a Killing submersion. 
In the case of a surface $S^{2}\imto N^3$ in a $3$-dimensional space, the decomposition of the bitension  field  \eqref{bitensionfield1}, with respect to its normal and tangential components  (see, for example, \cite{C1,Ou1}),  leads to the fact that $S^{2}$ is biharmonic into $N^3$ if and only if the mean curvature function $H=\trace A$ satisfies the system
	\begin{equation}\label{DecomposicaoBitensao}
	\begin{cases}
	\Delta H +H\, |A|^2-H\, \ri(\eta,\eta)=0\\
	2 A (\grad{H})+ H\, \grad{H}- 2H\, \ri(\eta)^{\top}=0,
	\end{cases}
	\end{equation}
where $A$ is the shape operator and $\ri (\eta)^{\top}$ is the tangent component of the Ricci curvature of $N$ in the direction of  the vector field $\eta$.

Let now $S$  be an oriented, simply connected surface of a Killing submersion $\E(G,r)$ with constant mean curvature. Without loss of generality, since our study is local, we can identify $\E(G,r)$ with the canonical model described in Example~\ref{defi:canonical}. Let $\{e_1,e_2,\eta\}$ be a local frame adapted to the surface $S$ (where  $\{e_1,e_2\}$ are not necessarily the frame defined in \eqref{basee_i}). Then, with respect to the orthonormal frame  \eqref{eqn:base-universal-killing}, we have
 \begin{equation}\label{basis}
e_1=\sum a_i\,E_i,\qquad e_2=\sum b_i\,E_i,\qquad \eta=\sum c_i\,E_i,
\end{equation} 
where $a_i, b_i, c_i, i=1,2,3$, are smooth functions locally defined on $S$. 

Using the above adapted frame, the condition of biharmonicity of the surface $S$ can be explicitly formulated as described in the following proposition.
\begin{proposition}\label{propsisp}
A CMC surface $S$ in a Killing submersion $\E(G,r)$ is proper biharmonic if and only if 
$H\neq 0$ and the components given in \eqref{basis} of the local frame $\{e_1,e_2,\eta\}$ 
satisfy
\begin{equation}\label{sistemap}
\left\{\begin{aligned}
&(G-2r^2)+(4r^2-G)\,c_3^2+\frac{2c_3\,(c_2\,r_x-c_1\,r_y)}{\lambda}-|A|^2=0,\\
&(4r^2-G)\,c_3\,a_3+\frac{[r_x\,(c_2\,a_3+c_3\,a_2)-r_y\,(c_1\,a_3+c_3\,a_1)]}{\lambda}=0,\\
&(4r^2-G)\,c_3\,b_3+\frac{[r_x\,(c_2\,b_3+c_3\,b_2)-r_y\,(c_1\,b_3+c_3\,b_1)]}{\lambda}=0.\\
\end{aligned}
\right.
\end{equation}
\end{proposition}
\begin{proof}
A straightforward computation, using \eqref{ricci1} and the basis \eqref{basis}, gives 
\begin{equation}\label{tangentR}
\begin{aligned}
(\ri(\eta))^{T}=&\sum_{i=1}^2\ri(\eta,e_i)\,e_i\\=&\bigg((4r^2-G)\,c_3\,a_3+\frac{[r_x\,(c_2\,a_3+c_3\,a_2)-r_y\,(c_1\,a_3+c_3\,a_1)]}{\lambda}\bigg)\,e_1\\
&+\bigg((4r^2-G)\,c_3\,b_3+\frac{[r_x\,(c_2\,b_3+c_3\,b_2)-r_y\,(c_1\,b_3+c_3\,b_1)]}{\lambda}\bigg)\,e_2\\
:=&\tilde{A}\,e_1+\tilde{B}\,e_2
\end{aligned}
\end{equation}
and
\begin{equation}\label{etaeta}
\begin{aligned}
\ri(\eta,\eta)&=(G-2r^2)+(4r^2-G)\,c_3^2+\frac{2c_3\,(c_2\,r_x-c_1\,r_y)}{\lambda}.
\end{aligned}
\end{equation} 
So, from \eqref{DecomposicaoBitensao}, it follows that the surface $S$ is proper biharmonic if and only if the system 
\begin{equation*}
\left\{\begin{aligned}
&(G-2r^2)+(4r^2-G)\,c_3^2+\frac{2c_3\,(c_2\,r_x-c_1\,r_y)}{\lambda}-|A|^2=0,\\
&\tilde{A}=0,\\
&\tilde{B}=0,\\
\end{aligned}
\right.
\end{equation*}
is satisfied. 
\end{proof}
\begin{remark}\label{zero}
Note that there is no proper biharmonic CMC surface $S$ in a Killing submersion $\E(G,r)$ with  $\phi\equiv 0$  on a open subset of $S$. In fact, if $\phi\equiv 0$, we have that $\xi=\eta$ (i.e. $c_3=1$) on the open subset. Therefore, 
  the distribution determined by $E_1$ and $E_2$ should be integrable and, hence,  involutive.
		As $$[E_1,E_2]=\frac{\lambda_x}{\lambda^2}\,E_1+\frac{\lambda_y}{\lambda^2}\,E_2+2r\,E_3,$$ it follows that $r=0$. Substituting in the first equation of \eqref{sistemap}, we have that $|A|=0$,
		i.e. the surface is totally geodesic and, consequently, minimal. 
\end{remark}
\begin{remark}
From Proposition~\ref{propsisp} and 
$$
\grad(r)=E_1(r)E_1+E_2(r)E_2=\frac{r_x}{\lambda}\,E_1+\frac{r_y}{\lambda}\,E_2,
$$
we deduce that if $S$ is a proper biharmonic surface in $\E(G,r)$, then
\begin{equation}\label{nulla}
0=\tilde{A}\,b_3-\tilde{B}\,a_3=\frac{c_3\,(c_1\,r_x+c_2\,r_y)}{\lambda}=\cos\phi\, \langle \grad(r), \eta\rangle.
\end{equation}
 Now,
assume that $\phi\neq\pi/2,0$ on a open set of $S$ and that $\grad(r)\neq 0$ on the same open set. Then, $\grad(r)$ is locally tangent to the surface $S$.  Consequently, we can choose the local orthonormal adapted frame $\{e_1,e_2,\eta\}$  such that $\{e_1,e_2\}$ is the frame defined in \eqref{basee_i} and
\begin{equation}\label{base-b1}
e_2=\frac{\grad(r)}{|\grad(r)|}.
\end{equation}
In fact, with this choice for $e_2$, using the notation of  \eqref{basis},  we have:
$$b_1=\frac{r_x}{\lambda\,|\grad(r)|},\quad b_2=\frac{r_y}{\lambda\,|\grad(r)|},\quad b_3=0,$$  Moreover, since $b_3=0$,
$a_3^2+c_3^2=1$, as $c_3=\cos\phi$, we must have that $a_3=\sin\phi$.  Then, it is easy to check that: 
\begin{equation}\label{base-b2}
\left\{\begin{aligned}
e_1&=-\cos\phi\,Je_2+\sin\phi\,\xi,\\
\eta&=\sin\phi\,Je_2+\cos\phi\,\xi,\\
\end{aligned}
\right.
\end{equation}
from which it follows that $\sin\phi \,e_1=-\eta \cos\phi +\xi$ and, consequently, $\sin\phi\, e_1=T$.
Finally, 
system~\eqref{sistemap} becomes
\begin{equation}\label{eq-due}
\left\{\begin{aligned}
&G-2r^2+(4r^2-G)\,\cos^2\phi+|\grad(r)|\,\sin (2\phi)-|A|^2=0,\\
&(G-4r^2)\frac{\sin(2\phi)}{2}+|\grad(r)|\,\cos(2\phi)=0.\\
\end{aligned}
\right.
\end{equation}
\end{remark}

As a first consequence of \eqref{eq-due}, we have the following non existence result of proper biharmonic CMC surfaces.

\begin{proposition}\label{propRconst}
Let $S$ be a CMC surface in a Killing submersion $\E(G,r)$ and assume that $G-4r^2\equiv0$ and $|\grad (r)|\neq 0$ on $S$. Then, $S$ cannot be proper biharmonic. 
\end{proposition}
\begin{proof}
Under the hypothesis, the second equation of system~\eqref{eq-due} implies that $\cos(2\phi)=0$ 
(i.e. $\phi\in\{\pi/4,3\pi/4\}$),
and the first equation of the system~\eqref{eq-due} reduces to
$$
|A|^2=2r^2\pm|\grad(r)|.
$$
On the other hand, from \eqref{A1} with $\phi$ constant, we have 
$$
|A|^2=2r^2+\mu^2, \qquad H=\mu\neq 0.
$$
Then, the case $\phi=3\pi/4$ does not occur and $|\grad(r)|=H^2=\text{constant}$. Since 
$e_1(r)=0$ and $e_2(r)=|\grad(r)|$, we conclude, using $e_1\,e_2(r)-e_2\,e_1(r)=(\nabla_{e_1}e_2-\nabla_{e_2}e_1)(r)$ and \eqref{nabla-ei}, that
	\begin{equation}\label{e1grad}
	0=e_1(|\mathrm{grad(r)}|)=- \mu\,|\grad(r)|,
	\end{equation}
	contradicting the hypothesis that $|\grad (r)|\neq 0$.
\end{proof}


\section{The proof of Theorem~\ref{main-theorem}}


\subsection{Proof of a)} Under the condition $\phi\equiv \pi/2$, system~\eqref{sistemap} becomes:
\begin{equation}\label{norma}
\left\{
\begin{aligned}
&|A|^2=G-2r^2,\\
&a_3\,(c_2\,r_x-c_1\,r_y)=0,\\
&b_3\,(c_2\,r_x-c_1\,r_y)=0.
\end{aligned}
\right.
\end{equation}
From the last two equations of \eqref{norma}, we obtain that
 \begin{equation}\label{orto}
 c_2\,r_x-c_1\,r_y=0.
 \end{equation}
Since $\eta$ is orthogonal to $\xi$, we can suppose that the basis $\{e_1,e_2,\eta\}$ is given by 
$$e_1=\tilde{a}\,E_1+\tilde{b}\,E_2,\qquad e_2=\xi,\qquad \eta=\tilde{b}\,E_1-\tilde{a}\,E_2,$$
 with $\tilde{a}^2+\tilde{b}^2=1$. To prove that the surface is locally a Hopf tube, it is sufficient to show that $e_1$ and $\eta$ are constant along the orbits of the Killing vector field $\xi$, that is $e_2(\tilde{a})=e_2(\tilde{b})=0$. To this end, we first observe that \eqref{eqn:killing-conexion} yields 
\begin{equation}\label{unoeta}
\begin{aligned}
\ln_{e_1}\eta&=\bigg(e_1(\tilde{b})+\frac{\tilde{a}\,(\tilde{b}\,\lambda_x-\tilde{a}\,\lambda_y)}{\lambda^2}\bigg)\,E_1\\&
+\bigg(-e_1(\tilde{a})+\frac{\tilde{b}\,(\tilde{b}\,\lambda_x-\tilde{a}\,\lambda_y)}{\lambda^2}\bigg)\,E_2-r\,E_3
\end{aligned}
\end{equation}
and
\begin{equation}\label{dueeta}
\ln_{e_2}\eta=e_2(\tilde{b})\,E_1-e_2(\tilde{a})\,E_2-r\,e_1.
\end{equation}
Also, the condition $\tilde{a}^2+\tilde{b}^2=1$ gives 
$$e_2(\tilde{a})=-\tilde{a}\tilde{b}\,e_2(\tilde{b})+\tilde{b}^2\,e_2(\tilde{a}), \qquad  e_2(\tilde{b})=-\tilde{a}\tilde{b}\,e_2(\tilde{a})+\tilde{a}^2\,e_2(\tilde{b})$$ 
and, so
$$e_2(\tilde{b})\,E_1-e_2(\tilde{a})\,E_2=(\tilde{a}\,e_2(\tilde{b})-\tilde{b}\,e_2(\tilde{a}))\,e_1.$$
Thus, substituting in \eqref{unoeta} and \eqref{dueeta}, it results that
\begin{equation}\label{eq-nabla-ei-eta}
\begin{aligned}
\ln_{e_1}\eta&=\bigg(\tilde{a}\,e_1(\tilde{b})-\tilde{b}\,e_1(\tilde{a})+\frac{\tilde{b}\lambda_x-\tilde{a}\lambda_y}{\lambda^2}\bigg)\,e_1-r\,e_2,\\
\ln_{e_2}\eta&=(\tilde{a}\,e_2(\tilde{b})-\tilde{b}\,e_2(\tilde{a})-r)\,e_1.
\end{aligned}
\end{equation}
By using \eqref{eq-nabla-ei-eta}, we obtain $$r=-\langle\ln_{e_1}\eta,e_2\rangle=-\langle\ln_{e_2}\eta,e_1\rangle=r-\tilde{a}\,e_2(\tilde{b})+\tilde{b}\,e_2(\tilde{a}).$$ The latter implies $\tilde{a}\,e_2(\tilde{b})+\tilde{b}\,e_2(\tilde{a})=0$
that, together with $\tilde{a}\,e_2(\tilde{a})+b\,e_2(\tilde{b})=0$, imply
$e_2(\tilde{a})=e_2(\tilde{b})=0$,
as desired. Then, the statement a) follows from Theorem~\ref{teo-hopf-tube}.

%
%
%
%
%
%
%
  \begin{remark}\label{r-G-constants}
If $\phi\neq \pi/2$, locally the surface can be considered as a graph over an open set $\Omega$  of the base space $M \subset \E(G,r)$:
  	$$S(x,y)=(x,y,z(x,y)), \qquad (x,y)\in \Omega.$$
As the functions $r$ and $G$ are constant with respect to the $z$-coordinate, we conclude that if $r$ and $G$ are constant along the surface, then they must be constant on a open set  of $\E(G,r)$. 
  \end{remark}
%

 \subsection{Proof of b1)}
 With the condition $r\equiv \text{constant}$ on $S$, 	system~\eqref{sistemap}  becomes
	\begin{equation}\label{sys-r-const}
	\left\{\begin{aligned}
	&(G-2r^2)+(4r^2-G)\,c_3^2-|A|^2=0,\\
	&a_3\,c_3\,(4r^2-G)=0,\\
	&b_3\,c_3\,(4r^2-G)=0.\\
	\end{aligned}
	\right.
	\end{equation}
Therefore, we have two possibilities.
	 \begin{enumerate}
		\item If $G\neq 4r^2$ over $S$, then $a_3=b_3=0$ and $|c_3|=1$. So $\text{Span}\{e_1,e_2\}=\text{Span}\{E_1,E_2\}$. Consequently, the distribution determined by $E_1$ and $E_2$ is integrable and, hence, is involutive.
		As $$[E_1,E_2]=\frac{\lambda_x}{\lambda^2}\,E_1+\frac{\lambda_y}{\lambda^2}\,E_2+2r\,E_3,$$ it follows that $r=0$. Thus, the first equation of \eqref{sys-r-const} implies that $|A|=0$,
		i.e. the surface is totally geodesic and, consequently, it's minimal.
		
		\item If $G=4r^2\equiv\text{constant}$ over $S$, from the Remark~\ref{r-G-constants} we conclude that $r$ and $G$ are constant on a open set $U$ of $\E(G,r)$. This implies that $U$ is an open set of a Bianchi-Cartan-Vranceanu space with $G=4r^2$. From the classification of the Bianchi-Cartan-Vranceanu spaces, we deduce that $U$ is an open set of either $\r^3$ or $\s^3$. Moreover, system~\eqref{sys-r-const} gives $|A|^2=2r^2$. Now the result announced in b1) follows immediately from the non existence of proper biharmonic surfaces in $\r^3$ (see \cite{CI}) and the classification of proper biharmonic surfaces in $\s^3$ given in  \cite{CMO}.  
	\end{enumerate}

\subsection{Proof of b2)}

First of all, with respect to the adapted local frame  $\{e_1,e_2,\eta\}$  described in \eqref{base-b1} and \eqref{base-b2}, the Gauss and Codazzi equations of Proposition~\ref{GaussCodazzi}, taking into account \eqref{base-b2}, are given, respectively,  by:
\begin{equation}\label{gauss2}
K= \det A-|A|^2+G-r^2,
\end{equation}
\begin{equation}\label{codazzi2}
\begin{aligned}
\nabla_{e_1} A (e_2) -\nabla_{e_2} A (e_1)=A[e_1,e_2] .
\end{aligned}
\end{equation}		


Now, according to Proposition~\ref{propRconst}, we can assume that $G-4r^2\neq 0$ on $S$.
%
The second equation of \eqref{eq-due} gives immediately that $\tan(2\phi)={2|\grad(r)|}/({4r^2-G})$, from which we can eliminate $G$ from the first equation of \eqref{eq-due} which becomes equivalent to
	\begin{equation}\label{e1-norma}
	\left\{\begin{aligned}
	&|A|^2=2r^2+|\grad(r)|\,\tan\phi,\\
   &\tan(2\phi)=\frac{2|\grad(r)|}{4r^2-G}.
   \end{aligned}
    \right.
    \end{equation}

Taking into account  \eqref{A1} and $\mu=H-e_1(\phi)$, we have that
\begin{equation}\label{e-norma}
|A|^2=2\,( e_1(\phi)^2+e_2(\phi)^2)+H^2+2 r^2-4r\, e_2(\phi)-2H\,e_1(\phi).
\end{equation}
A straightforward computation, using \eqref{A1} and \eqref{nabla-ei}, shows that the $e_2$ component of the Codazzi equation~\eqref{codazzi2} is given by:
\begin{equation}\label{codazzie_2}
\begin{aligned}
2\big(e_1(\phi)^2+e_2(\phi)^2\big)+H^2-3H\, e_1(\phi)-6r\, e_2(\phi)+4r^2&\\
-\tan(\phi)\big(e_2\,e_2(\phi)+e_1\,e_1(\phi)-|\grad(r)|\big)&=0.
\end{aligned}
\end{equation}
Then, combining \eqref{e-norma} and \eqref{codazzie_2}, we get
\begin{equation}\label{normaA}
|A|^2=\tan\phi\,\big(e_1e_1(\phi)
+e_2e_2(\phi)-|\grad(r)|\big)-2r^2+2 r\, e_2(\phi)+H\,e_1(\phi),\end{equation}
that, together with \eqref{e1-norma}, turns over 
\begin{equation}\label{co1}
2\,|A|^2=\tan\phi\,(e_1e_1(\phi)+e_2e_2(\phi))+2 r\,e_2(\phi)+H\,e_1(\phi).\end{equation}
So, \eqref{A-phi} is obtained observing that
\begin{equation}\label{co2}
\begin{aligned}
\Delta(\phi)&=e_1e_1(\phi)+e_2e_2(\phi)-(\nabla_{e_1}e_1)(\phi)-(\nabla_{e_2}e_2)(\phi)\\
&=e_1e_1(\phi)+e_2e_2(\phi)-\cot(\phi)\,\big[|\grad(\phi)|^2-\big(2r\, e_2(\phi)+H\,e_1(\phi)\big)\big].
\end{aligned}
\end{equation}

\section{Biharmonic Hopf tubes}
Following \cite{Ou1}, let $S$ be an Hopf tube of a  Riemannian submersion $\pi:(N^3,g)\to (M^2,h)$ with totally geodesic fibers. Denote by $\alpha: I \to (M^2,h)$ the base curve of $S$, by $\beta:I \to (N^3,h)$ its horizontal lift and consider the Frenet frame of $\beta$ given by $\{e_1:=\beta',e_2,\eta\}$, with $\eta$ being the unit normal vector of $S$. Then, it is proved in \cite{Ou1} that $S$ is proper biharmonic if and only if 
	\begin{equation}\label{sys-Ou}
	\left\{\begin{aligned}
	&\kappa''_g-\kappa_g(\kappa_g^2+2\tau_g^2)+\kappa_g\ri^N(\eta,\eta)=0,\\
	&3\kappa'_g\kappa_g-\kappa_g \ri^N(\eta,e_1)=0,\\
	&\kappa'_g \, \tau'_g+\kappa_g \ri^N(\eta,e_2)=0,	
	\end{aligned}
	\right.
	\end{equation}
	where $\kappa_g$ and $\tau_g$ are the geodesic curvature and geodesic torsion of $\beta$.
	
	Using \eqref{sys-Ou} we obtain the following classification of proper biharmonic Hopf tubes of a Killing submersion.
\begin{theorem}\label{teo-hopf-tube}
	Let $S$ be a Hopf tube of a Killing submersion $\E(G,r)$. Then, $S$ is proper biharmonic if and only if $G$, $r$ and the mean curvature function $H$ are constant along $S$ and they satisfy
$$
H^2=G-4r^2.
$$
\end{theorem}
\begin{proof}
	Let $\alpha(s)=(x(s),y(s))$ be a curve parametrized by arc length in $M^2$, with geodesic curvature $\kappa_g$. Taking the horizontal lifts of the tangent and the normal vectors fields of $\alpha$ we have, respectively,
	$$e_1=\lambda\, x'\,E_1+\lambda\,y'\,E_2, \qquad \eta=\lambda\, y'\,E_1-\lambda\,x'\,E_2,$$ that, together with $e_2=E_3$, form an orthonormal frame adapted to the Hopf tube.
	
    In this basis, the normal and tangent components of  $\ri(\eta)$ are  
	\begin{equation}\label{ricci-2}
	\ri(\eta,\eta)=G-2r^2, \qquad
	\ri(\eta,e_1)=0, \qquad
	\ri(\eta,e_2)=-(x'\,r_x+y'\,r_y).
	\end{equation}
	Also, the geodesic torsion of the horizontal lift $\beta(s)$ of  $\alpha(s)$ is
	\begin{equation}\label{torsion}
	\tau_g(s)=-\langle\ln_{e_1} e_2,\eta\rangle=-r(\alpha(s)),
	\end{equation}
	and, since $\ln_{e_2}e_2=0$, the mean curvature function of $S$ is given by
	\begin{equation}\label{H}
          H=\langle\ln_{e_1}e_1, \eta\rangle=\kappa_ g.
	\end{equation}
	Substituting \eqref{ricci-2} and \eqref{torsion} in \eqref{sys-Ou}, we obtain
	\begin{equation}\label{last}
	\left\{\begin{aligned}
	&\kappa_g''-\kappa_g^3+(G-4r^2)\,\kappa_g=0\\
	&\kappa_g\,\kappa_g'=0\\
	&r\,\kappa_g'+(x'\,r_x+y'\,r_y)\,\kappa_g=0\\
	\end{aligned}
	\right.
	\end{equation}
	Therefore, from the second equation in \eqref{last}, since $S$ is proper, we deduce that $k_g=c\neq 0$, $c\in\r$.  Moreover the first equation gives $c^2=G-4r^2$. Thus, the Hopf tube is proper biharmonic if and only if
	\begin{equation}
	\left\{\begin{aligned}
	&H^2=G-4r^2,\\
	&(x'\,r_x+y'\,r_y)=0.\\
	\end{aligned}
	\right.
	\end{equation}
	The result follows observing that $(x'\,r_x+y'\,r_y)=0$ is equivalent to $r\equiv \text{constant}$ along the curve $\alpha$ and, so, along the surface.
\end{proof}

\begin{example}
Take $\Omega = (\varepsilon_1,\varepsilon_2)\times (0,2\pi)$. Consider a constant $r$ and any smooth function $f:(\varepsilon_1,\varepsilon_2)\rightarrow \mathbb{R}$ such that there exists $t_0\in (\varepsilon_1,\varepsilon_2)$ with
$f \left[f''(t_0)+4 r^2 f(t_0)\right]+\left(f'(t_0)\right)^2=0$.  Define on $\Omega$ the metric $ds^2=dt^2+f^2(t)d\vartheta^2$ and the curve $\gamma (s):=(t_0, {s}/{f(t_0)}+\mu),\, \mu\in \mathbb{R}$. Then, $\gamma$ is a curve in $M=(\Omega,ds^2)$ with constant geodesic curvature $\kappa_g=f'(t_0)/f(t_0)$ and the Gaussian curvature of $M$ is constant along $\gamma$ given by $G=-{f''(t_0)}/{f(t_0)}$.  Using Theorem~\ref{teo-hopf-tube} we conclude that the Hopf cylinder shaped on $\gamma$ is a proper biharmonic surface in $\mathbb{E}({G}, {r})$.
\end{example}


\begin{thebibliography}{999}
\bibitem{PBJCW} P.~Baird, J.C.~Wood. {\it Harmonic Morphisms between Riemannian
Manifolds}. Oxford Science Publications, (2003).

\bibitem{BarGar-preparation} M. Barros, O.J. Garay. Willmore-like tori in Killing submersions, in preparation.

\bibitem{CMO} R.~Caddeo, S.~Montaldo, C.~Oniciuc. Biharmonic submanifolds of $\s^3$. {\em Int. J. Math.} 12 (2001), 867--876.


\bibitem{C1} B.-Y.~Chen. Total mean curvature and submanifolds of finite type, second edition. {\em Serier in Pure Mathematics, World Scentic}, Vol. 27 (2015).
	
\bibitem{CI} B.-Y.~Chen, S.~Ishikawa. Biharmonic pseudo-Riemannian submanifolds in pseudo-Euclidean spaces. {\em Kyushu J.~Math} 52 (1998), 167--185.	
	
\bibitem{Da} B.~Daniel. Isometric immersions into $3$-dimensional homogeneous manifolds. {\em Comment. Math. Helv.}  82 (2007), 87--131.	
	
\bibitem{ES} J.~Eells, J.H.~Sampson. Harmonic mappings of Riemannian manifolds. {\em Amer. J. Math.} 86 (1964), 109--160.

\bibitem{EL} J.~Eells, L.~Lemaire. Selected Topics in Harmonic Maps. {\em Bull. London Math. Soc.} 20 (1988), 385--524.


\bibitem{Espinar}  J.M.~Espinar and I.S.~de Oliveira. Locally convex surfaces immersed in a Killing submersion.
{\em Bull. Braz. Math. Soc.} 44 (2013), 155--171.

\bibitem{J1} G.Y.~Jiang. $2$-harmonic maps and their first and second variation formulas. {\em Chinese Ann. Math. Ser.} A~7, 7 (1986), 130--144.

\bibitem{Manzano}  J.M.~Manzano. On the classification of Killing submersions and their isometries. {\em Pacific. J. Math.}  270 (2014), 367--392.

\bibitem{Montaldo} S.~Montaldo, C.~Oniciuc. {A short survey on biharmonic maps between riemannian manifolds}. {\it Rev. Un. Mat. Argentina}, 47 (2006), 1--22.

\bibitem{Ou1} Y.-L.~Ou. Biharmonic hypersurfaces in Riemannian manifolds. {\em Pacific J. of Math.} 248, No. 1 (2010), 217--232.

\bibitem{OW} Y.-L. Ou, Ze-Ping Wang.
Constant mean curvature and totally umbilical biharmonic surfaces in $3$-dimensional geometries,  {\em J. Geom. Phys.} 61 (2011), 1845--1853.

\bibitem{Xin} Y.~Xin. {Geometry of harmonic maps}. {\it Progress in Nonlinear Differential Equations and their Applications}, Birkh\"{a}user Boston Inc., Boston (1996).

\end{thebibliography}
\end{document}